\documentclass[10pt]{amsart}

\usepackage{lmodern}
\usepackage[T1]{fontenc}

\usepackage{amsmath, amsthm, amscd, amssymb, xcolor}
\definecolor{dblue}{rgb}{0,0,.6}
\usepackage[colorlinks=true, linkcolor=dblue, citecolor=dblue, filecolor = dblue, menucolor = dblue, urlcolor = dblue]{hyperref}
\usepackage{tikz-cd} 

\usepackage[all,cmtip]{xy}

\newcommand{\bbZ}{\mathbb{Z}}
\newcommand{\bbQ}{\mathbb{Q}}
\newcommand{\bbR}{\mathbb{R}}
\newcommand{\bbC}{\mathbb{C}}

\newcommand{\bbP}{\mathbb{P}}

\newcommand{\DD}{\mathcal{D}}
\newcommand{\FF}{\mathcal{F}}
\newcommand{\HH}{\mathcal{H}}

\newcommand{\NN}{\mathcal{N}}
\newcommand{\OO}{\mathcal{O}}
\newcommand{\PP}{\mathcal{P}}

\newcommand{\VV}{\mathcal{V}}

\newcommand{\XX}{\mathcal{X}}
\newcommand{\YY}{\mathcal{Y}}

\newcommand{\frh}{\mathfrak{h}}

\newcommand{\frsl}{\mathfrak{sl}}

\newcommand{\frso}{\mathfrak{so}}
\newcommand{\GL}{\mathrm{GL}}

\renewcommand{\ge}{\geqslant}
\renewcommand{\le}{\leqslant}

\newcommand{\st}{\enskip |\enskip}
\newcommand{\sdot}{{\raisebox{0.16ex}{$\scriptscriptstyle\bullet$}}}
\newcommand{\Tr}{\mathrm{Tr}}

\newcommand{\emrp}{\mathrm{End}}

\newcommand{\ii}{i}

\newcommand{\Cl}{\mathcal{C}l}

\newcommand{\Spin}{\mathrm{Spin}}
\newcommand{\lrarr}{\longrightarrow}
\newcommand{\hrarr}{\hookrightarrow}

\newtheorem{defn}{Definition}[section]
\newtheorem{prop}[defn]{Proposition}
\newtheorem{thm}[defn]{Theorem}
\newtheorem{lem}[defn]{Lemma}
\newtheorem{cor}[defn]{Corollary}

{\theoremstyle{remark}
  \newtheorem{rem}[defn]{Remark}
  
}

\makeatletter
\@addtoreset{equation}{section}
\makeatother

\title{Limit mixed Hodge structures of hyperk\"ahler manifolds}

\author{Andrey Soldatenkov}
\address{Institut f\"ur Mathematik, Humboldt-Universit\"at zu Berlin, Unter den Linden 6, 10099 Berlin}
\email{soldatea@hu-berlin.de}
\date{\today}

\thanks{The author was supported by the SFB/TR 45 `Periods, Moduli Spaces and Arithmetic of Algebraic Varieties'
of the DFG (German Research Foundation).}

\begin{document}

\begin{abstract}
This note is inspired by the work of Deligne \cite{Del1}. We study limit mixed Hodge structures
of degenerating families of compact hyperk\"ahler manifolds.
We show that when the monodromy action on $H^2$ has maximal index of unipotency,
the limit mixed Hodge structures on all cohomology groups are of Hodge-Tate type.
\end{abstract}

\maketitle

\section{Introduction} 

It is well-known that for the study of mirror symmetry it is important
to consider families of Calabi-Yau varieties with ``maximal degeneration''
at the special fibre. There are several slightly different ways to give
the definition of maximal degeneration (cf. \cite[Definition 3]{Mor}, \cite[Definition 1]{KS}, \cite{Del1}),
not all of them being equivalent to each other. In any of these definitions the
condition is Hodge-theoretic, and concerns the limiting behavior
of the corresponding variation of Hodge structures.
Presumably the strongest condition was suggested by Deligne \cite{Del1}:
a degeneration of Hodge structures is called maximal, if the corresponding limit mixed Hodge
structure is of Hodge-Tate type, i.e. it is an iterated
extension of direct sums of $\bbZ(k)$, $k\in \bbZ$. 

We study projective degenerations of compact simply-connected hyperk\"ahler manifolds
over the unit disc (see Definition \ref{def_deg}). The main result (Theorem \ref{thm_main}) states
the following: if the monodromy operator $\gamma$ acting on $H^2$ is unipotent of maximal index,
i.e. $(\gamma - \mathrm{id})^2 \neq 0$ and $(\gamma - \mathrm{id})^3 = 0$, then the limit
mixed Hodge structures on $H^k$ are of Hodge-Tate type for all $k$. We deduce this result from the
generalized Kuga-Satake construction \cite{KSV}, \cite{SS}. The condition of maximal
unipotency of monodromy is a priory weaker than maximality in the sense of Deligne.
Our result shows that these two notions coincide in the case of hyperk\"ahler manifolds.

The key step for understanding the limit mixed Hodge structures of degenerations is the
description of the monodromy action on the cohomology ring. Using the results
of Verbitsky \cite[Theorem 3.5(iii)]{Ve3}, we show that monodromy
action on the full cohomology ring is essentially determined by its action on $H^2$
(see Proposition \ref{prop_mon}).

In the case of maximal degenerations of compact hyperk\"ahler manifolds,
one can determine the unipotency indices of the monodromy action on $H^k$.
For even $k$ this was done 
in \cite[Proposition 6.18]{KLSV}, see also the paper of Nagai \cite{Nag}
for the general discussion and related results.
We compute the unipotency indices for odd $k$, see Proposition \ref{prop_index}.
This result applies to degenerations of generalized Kummer type manifolds, since
they have non-trivial cohomology groups in odd degrees.

In section \ref{sec_existence} we discuss existence of maximal degenerations, showing that such degenerations
exist in every deformation equivalence class of compact hyperk\"ahler manifolds with $b_2\ge 5$ (Theorem \ref{thm_existence}).
This result has already appeared in the preprint \cite{To}. We provide a simple independent argument,
showing that one can always find a nilpotent orbit (see Definition \ref{def_nilp}) with maximally unipotent monodromy that is
induced by a projective degeneration of hyperk\"ahler manifolds.

\section{Degenerations with maximally unipotent monodromy}

In this section we recall some well-known facts about compact hyperk\"ahler manifolds
and their period domains, for an overview see \cite{Hu1}. We also recall necessary
facts about degenerations and limit mixed Hodge structures.

\subsection{Hyperk\"ahler manifolds}\label{sec_notations}

Recall that a compact K\"ahler manifold $X$ is called simple hyperk\"ahler, or
irreducible holomorphic symplectic (IHS), if it is simply-connected and $H^0(X,\Omega^2_X)$
is spanned by a symplectic form. In what follows we will always assume that
$X$ is simple hyperk\"ahler of complex dimension $2n$.

Let $V_\bbZ = H^2(X,\bbZ)$ and $V = V_\bbZ\otimes \bbQ$. Note that $V_\bbZ$ is torsion-free,
because $X$ is simply-connected. Recall that there exists a non-degenerate
form $q\in S^2V^*$ and a constant $c_X\in \bbQ$, such that for all
$h\in H^2(X,\bbQ)\simeq V$ we have $q(h)^n = c_X h^{2n}$, where we use the cup product in cohomology.
The form $q$ is called Beauville-Bogomolov-Fujiki (BBF) form. We normalize $q$ to
make it integral and primitive on $V_\bbZ$, and such that $q(h) > 0$ for a K\"ahler
class $h$. Then $q$ has signature $(3,b_2(X)-3)$.

Let $\hat{\DD} \subset \bbP(V_\bbC)$ be the quadric defined by $q$, and $\DD =\{x\in \hat{\DD}\st q(x,\bar{x})>0\}$.
Given an element $h\in V_\bbZ$ with $q(h)>0$ we will denote: $V^h = \{v\in V\st q(h,v) = 0\}$,
$\hat{\DD}^h = \hat{\DD}\cap \bbP(V^h_\bbC)$, $\DD^h = \DD\cap \bbP(V^h_\bbC)$.
Then $\hat{\DD}^h$ is the extended period domain and $\DD^h$ is the period domain
for polarized $\bbQ$-Hodge structures of K3 type on $(V^h,q)$.

The group $G_\bbR = \mathrm{O}(V^h_\bbR,q)$ acts transitively on $\DD^h$. After fixing a base point in $\DD$, we get an
isomorphism $\DD^h\simeq G_\bbR/K$, where $K$ is a compact subgroup. Analogously, $G_\bbC = \mathrm{O}(V^h_\bbC,q)$
acts transitively on $\hat{\DD}^h$.  The discrete
group $\mathrm{O}(V^h_\bbZ,q)\subset G_\bbR$ acts on $\DD^h$ properly discontinuously,
and according to Baily-Borel the quotient $\DD^h/\mathrm{O}(V^h_\bbZ,q)$ is a quasi-projective variety.
We can pass to a finite index torsion-free subgroup $\Gamma\subset \mathrm{O}(V^h_\bbZ,q)$, so that
the quotient $\DD^h/\Gamma$ is moreover smooth.

\subsection{Degenerations} Denote: $\Delta = \{z\in\bbC\st |z|<1\}$, $\Delta^* = \Delta\backslash \{0\}$.
Given a morphism $\pi\colon \XX\to \Delta$ and $t\in \Delta$ we write
$\XX_t = \pi^{-1}(t)$.

\begin{defn}\label{def_deg}
A degeneration of $X$ is a flat proper morphism of complex-analytic spaces
$\pi\colon \XX\to \Delta$, such that: $\pi$ is smooth over $\Delta^*$; the fibre $\XX_t$ is
deformation equivalent to $X$ for all $t\in \Delta^*$;
the monodromy action on the second cohomology of $\XX_t$ is unipotent and non-trivial.
The degeneration is called projective, if $\pi$ is a projective morphism.
\end{defn}

\begin{rem} The condition of unipotency is almost automatic: it follows from a theorem of
Borel (see \cite[Lemma 4.5]{Sch}), that monodromy of any family becomes unipotent
after we pass to a finite ramified cover of $\Delta$. Non-triviality of monodromy
excludes the case when $\pi$ is smooth over the whole $\Delta$.
Note that we do not require any of the smooth fibres of $\pi$ to be isomorphic to $X$,
but only deformation-equivalent to it. One may think that our degenerations represent
``boundary points'' of the connected component of the moduli space that contains $X$.
\end{rem}

Let $\pi\colon\XX\to \Delta$ be a projective degeneration of $X$. Denote by $\pi'$
the restriction of $\pi$ to $\pi^{-1}(\Delta^*)$ and consider the local system $\VV=R^2\pi'_*\bbZ$
over $\Delta^*$. Fix a base point $t\in \Delta^*$, identify $H^2(\XX_t,\bbZ)$ with
$V_\bbZ$ and let $h\in V_\bbZ$ be the class of the polarization. Then $\VV$ is a variation
of Hodge structures (VHS) with fibre $V_\bbZ$, $h$ determines a sub-VHS in it, and
$q$ defines a bilinear pairing on $\VV$. Let $\VV^h$ be the $q$-orthogonal complement
of $h$; it is a VHS with fibre $V^h_\bbZ$ polarized by $q$.

Let $\tilde{\Delta} = \{z\in \bbC\st \mathrm{Im}(z)>0\}$ and $\tau\colon \tilde{\Delta}\to\Delta^*$,
$z\mapsto e^{2\pi \ii z}$ be the universal covering. The pull-back $\tau^*\VV^h_\bbZ$ is a trivial
local system and the VHS on it defines a period map $\tilde{\varphi}\colon \tilde{\Delta}\to \DD^h$.
By our definition of degeneration, the monodromy transformation $\gamma\in \mathrm{Aut}(\VV^h,q) \simeq \mathrm{O}(V_\bbZ^h,q)$
is of the form $\gamma = e^N$, where $N\in \frso(V^h,q)$ is nilpotent of index $2$ or $3$.
This restriction on the index of nilpotency follows from the general statement \cite[Theorem 6.1]{Sch}.

\begin{defn}
Degenerations of $X$ with $N$ of index $3$ will be called maximally unipotent.
\end{defn}

\begin{rem} There exists a different terminology, used mainly in the
case of degenerations of K3 surfaces: the degeneration is of ``type II''
and ``type III'' when $N$ has nilpotency index $2$, respectively $3$,
see e.g. \cite{Ku}.
\end{rem}

We recall the results of Schmid \cite{Sch} about limit mixed Hodge structures (MHS).
The period map $\tilde{\varphi}$ satisfies the relation $\tilde{\varphi}(z+1) = \gamma\, \tilde{\varphi}(z)$.
Define the map $\tilde{\psi}\colon \tilde{\Delta}\to \hat{\DD}^h$, $z\mapsto e^{-zN}\tilde{\varphi}(z)$.
Then $\tilde{\psi}(z+1) = \tilde{\psi}(z)$, and $\tilde{\psi}$ descends to a map
$\psi\colon \Delta^*\to \hat{\DD}^h$. According to the nilpotent orbit theorem \cite[Theorem 4.9]{Sch},
$\psi$ extends over the puncture, and the point $\psi(0)$ determines a decreasing
filtration $F^\sdot_{\mathrm{lim}}$ on $V_\bbC^h$. Another filtration $W_\sdot$, increasing and defined over $\bbQ$,
is induced on $V^h$ by the nilpotent operator $N$. It follows from the $SL_2$-orbit theorem
that these two filtrations and the form $q$ determine a polarized mixed Hodge structure on $V^h$, see \cite[Theorem 6.16]{Sch}.

\begin{defn}
A mixed Hodge structure $(U,W_\sdot,F^\sdot)$ is of Hodge-Tate type, if its Hodge numbers
satisfy the condition $h^{p,q} = 0$ for $p\neq q$. Equivalently, $\mathrm{gr}^W_{2p+1}U = 0$,
and $\mathrm{gr}^W_{2p}U$ is a pure Hodge structures of type $(p,p)$ for all $p$.
\end{defn}

It is easy to check (see e.g. \cite{Ku}) that for a maximally unipotent degeneration of $X$
the limit MHS $(V^h,W_\sdot,F^\sdot_{\mathrm{lim}})$ on the second cohomology is of Hodge-Tate
type.

This finishes the discussion of the limit MHS on the second cohomology of $X$. Next, one can apply
the above constructions to higher degree cohomology groups. To study their behavior
under degeneration, we use the relation between Hodge structures on higher cohomology groups and on $H^2$.
This will be explained in the next section.

\section{Limit mixed Hodge structures of maximally unipotent degenerations}

In this section we fix $X$, $V_\bbZ$, $V$ and $q$ as above. We consider a projective
degeneration $\pi\colon\XX\to\Delta$ of $X$, and assume without loss of generality
that $X\simeq \XX_t$ for a fixed base point $t\in\Delta^*$. We let $h\in V_\bbZ$ be
the class of the polarization.

\subsection{The Mukai extension and the mapping class group.}\label{subsec_mukai}
Consider the graded $\bbQ$-vector space $\tilde{V} = \langle e_0\rangle\oplus V\oplus \langle e_4\rangle$,
where $e_i$ is of degree $i$, and $V$ is in degree 2. We introduce on $\tilde{V}$
a quadratic form $\tilde{q}$ that is determined by the following conditions:
$\tilde{q}|_V = q$, $e_0$ and $e_4$ are isotropic and orthogonal to $V$ and span a hyperbolic
plane, so that $\tilde{q}(e_0,e_4) = 1$. We call $(\tilde{V},\tilde{q})$ the Mukai
extension of $(V,q)$.

Consider the graded Lie algebra $\frso(\tilde{V},\tilde{q})$ and denote by $\Xi$
the generator of the orthogonal algebra of $\langle e_0, e_4\rangle$, such that $\Xi e_4 = e_4$, $\Xi e_0 = -e_0$.
Denote by $W$ the Weil operator that induces the Hodge structure on $V$,
i.e. it acts on $V^{p,q}$ as multiplication by $\ii(p-q)$.
It is clear that $\Xi, W\in \frso^0(\tilde{V},\tilde{q})$.

We recall that there exists a representation of graded Lie algebras
$\frso(\tilde{V},\tilde{q})\to \emrp(H^\sdot(X,\bbQ))$, such that:
the action of $\Xi$ induces the cohomological grading on $H^\sdot(X,\bbQ)$;
the action of $W$ induces the Hodge structures on $H^k(X,\bbQ)$ for all $k$.
For the proof we refer to \cite{Ve1}, \cite{LL} or \cite[Theorem A.10]{KSV}.

Recall also, that $\frso(V,q)$ acts on $H^\sdot(X,\bbQ)$ by derivations,
see \cite[Corollary 13.5]{Ve1}. It acts trivially on all Pontryagin classes of $X$,
since the Pontryagin classes stay of Hodge type $(p,p)$ on all deformations of $X$.
Denote by $\mathrm{Aut}^P(X)\subset \mathrm{GL}(H^\sdot(X,\bbQ))$ the group
of algebra automorphisms that fix the Pontryagin classes. We obtain a homomorphism
of algebraic groups $\alpha\colon\Spin(V,q)\to \mathrm{Aut}^P(X)$. Let us
denote by $\mathrm{Aut}^+(X)$ the image of $\alpha$.

It was shown in \cite{HS} (see also \cite{Hu2}), that $\int_X\sqrt{\mathrm{td}(X)}>0$.
Here $\mathrm{td}(X)$ denotes the total Todd class of $X$. Since all odd Chern classes of $X$
vanish, $\mathrm{td}(X)$ can be expressed as a universal polynomial in the Pontryagin classes.
It follows that all elements of $\mathrm{Aut}^P(X)$ act trivially on $H^{4n}(X,\bbQ)$,
where $2n = \mathrm{dim}_\bbC(X)$.

Consider the action of $\mathrm{Aut}^P(X)$ on $H^2(X,\bbQ)$. Note that the form $q$ is uniquely up to a sign
determined by the multiplicative structure of the cohomology ring. It follows from the above discussion
that the action of $\mathrm{Aut}^P(X)$ preserves $q$ (see \cite[Theorem 3.5(i)]{Ve3}).
Hence we have a homomorphism $\beta\colon\mathrm{Aut}^P(X)\to \mathrm{O}(V,q)$.

 We get the following commutative
diagram of algebraic groups, where the maps $\alpha'$ and $\beta'$ are isogenies:

\begin{equation}\label{groups}
\begin{tikzcd}[]
\mathrm{Spin}(V,q) \arrow[two heads]{dr}\arrow[two heads]{r}{\alpha'} &
\mathrm{Aut}^+(X)\arrow[two heads]{d}{\beta'}\arrow[hook]{r} & \mathrm{Aut}^P(X)\dar{\beta} \\
& \mathrm{SO}(V,q) \arrow[hook]{r} & \mathrm{O}(V,q)
\end{tikzcd}
\end{equation}

Given a $\bbQ$-algebraic group $G$ we will denote by $G_\bbQ$ the group of its rational points.

\begin{lem}\label{lem_mc}
Let $\Gamma_A$ be an arithmetic subgroup of $\mathrm{Aut}^P(X)_\bbQ$.
Then $\Gamma_A^+ = \Gamma_A \cap \mathrm{Aut}^+(X)_\bbQ$ is of finite index in $\Gamma_A$.
\end{lem}
\begin{proof} We first outline the idea of the proof. We do not know a priory that the subgroup
$\mathrm{Aut}^+(X)$ is of finite index in $\mathrm{Aut}^P(X)$, so we can not prove the statement
directly. We will instead consider the images of $\Gamma_A^+$ and $\Gamma_A$ under the maps $\beta'$
and $\beta$. These images are arithmetic subgroups of $\mathrm{O}(V,q)_\bbQ$, hence
they are commensurable, which is enough to deduce the claim of the lemma.

The group $\Gamma_A^+$ is an arithmetic subgroup of $\mathrm{Aut}^+(X)_\bbQ$, because $\Gamma_A$ is an
arithmetic subgroup of $\mathrm{Aut}^P(X)_\bbQ$.
Let us denote by $\Gamma_2$ and $\Gamma_2^+$ the images of $\Gamma_A$ in $\mathrm{O}(V,q)_\bbQ$
and of $\Gamma_A^+$ in $\mathrm{SO}(V,q)_\bbQ$ respectively. By \cite[Theorem 8.9]{Bo} $\Gamma_2^+$ 
is an arithmetic subgroup of $\mathrm{SO}(V,q)_\bbQ$. It is also an arithmetic subgroup
of $\mathrm{O}(V,q)_\bbQ$, since $\mathrm{SO}(V,q)_\bbQ$ is an index two subgroup in it.
By \cite[Corollary 7.13]{Bo} $\Gamma_2$ is contained in an arithmetic subgroup of $\mathrm{O}(V,q)_\bbQ$.
Since $\Gamma_2$ also contains an arithmetic subgroup $\Gamma_2^+$, it is arithmetic
itself.

We will use the following observation. Let $\phi\colon G_1\to G_2$ be a surjective homomorphism
of groups with finite kernel, and let $H\subset G_1$ be a subgroup. 
If $\phi(H)$ is of finite index in $G_2$, then $H$ is of finite index in $G_1$.
Let us apply the observation to $G_1= \Gamma_A$, $G_2 = \Gamma_2$ and $H = \Gamma_A^+$.
The proof of \cite[Theorem 3.5(iii)]{Ve3} shows that the kernel of $\beta\colon \Gamma_A \to \Gamma_2$ is finite
(note that the cited proof applies to arbitrary arithmetic subgroups of $\mathrm{Aut}^P(X)_\bbQ$).
The subgroup $\beta(\Gamma_A^+) = \Gamma_2^+$ has finite index in $\Gamma_2$, because both are
arithmetic subgroups of $\mathrm{O}(V,q)_\bbQ$. We conclude that $\Gamma_A^+$ is of finite index in $\Gamma_A$.
\end{proof}

\begin{rem} The lemma can be applied to the arithmetic subgroup $\Gamma_A = \mathrm{Aut}^P(X)_\bbZ$
of automorphisms that preserve the integral cohomology classes. Let
$\mathrm{MC}(X) = \mathrm{Diff}(X)/\mathrm{Diff}^0(X)$ be the mapping class group of $X$.
Here $\mathrm{Diff}(X)$ is the group of diffeomorphisms of $X$, and $\mathrm{Diff}^0(X)$ is
the subgroup of diffeomorphisms isotopic to the identity. 
We have a natural homomorphism $\mathrm{MC}(X) \to \mathrm{Aut}^P(X)_\bbZ$.
The above lemma shows that the action of $\mathrm{MC}(X)$ on the cohomology algebra
can essentially be recovered from its action on $H^2$, up to some elements of finite order,
bounded by the index of $\mathrm{Aut}^P(X)_\bbZ^+$ in $\mathrm{Aut}^P(X)_\bbZ$.
\end{rem}

\subsection{The Kuga-Satake construction} We recall the main result of \cite{KSV}.
To a Hodge structure $V$ of K3 type one
can associate the Kuga-Satake Hodge structure of abelian type.
It is constructed as follows. Let $H = \Cl(V,q)$ be the Clifford algebra
and let $v\in V_\bbC$ be the generator of $V^{2,0}$. Define $H^{0,-1}$ to
be the right ideal $vH_\bbC$ (see \cite[Lemma 3.3]{SS}), and let $H^{-1,0} =\overline{H^{0,-1}}$.
One can check that this defines a Hodge structure on $H$. Let $H^h$ be analogously defined
Hodge structure for $V^h$. Then $H\simeq (H^h)^{\oplus 2}$, and
one can check that $H$ is polarized, although the polarization is not canonical.
More precisely, fix a pair of elements $a_1,a_2\in V^h$, such that $q(a_1)>0$,
$q(a_2)>0$, $q(a_1,a_2)=0$. Let $a = a_1a_2\in \Cl(V^h,q)$ and $\omega(x,y) = \Tr(xa\bar{y})$,
where $x,y\in \Cl(V^h,q)$, the map $y\mapsto \bar{y}$ is the canonical anti-involution,
and $\Tr$ is the trace on the Clifford algebra (see e.g. \cite[Proposition 4.2]{KSV}).
Then either $\omega$ or $-\omega$ defines a polarization of $H^h$, moreover $\omega$ is $\Spin(V^h,q)$-invariant.
When we apply the Kuga-Satake construction to a VHS of K3 type, the monodromy operator
lies in $\Spin(V^h,q)$ (see \cite[Section 3.1]{SS}), hence the form $\omega$ is always monodromy-invariant.

Note that $H$ is canonically an $\frso(V,q)$-module, and the Hodge structure on it is induced
by the action of the Weil operator $W$. The following theorem was proved in \cite[Theorem 4.1]{KSV}

\begin{thm}\label{thm_ks}
There exists a structure of graded $\frso(\tilde{V},\tilde{q})$-module on $\Lambda^\sdot H^*$
that extends the canonical $\frso(V,q)$-module structure. Moreover, there exists
an integer $m>0$ and an embedding of $\frso(\tilde{V},\tilde{q})$-modules
\begin{equation}\label{eqn_ks}
H^{\sdot+2n}(X,\bbQ)\hrarr \Lambda^{\sdot+2d}(H^{* \oplus m}),
\end{equation}
where $2n = \dim_\bbC(X)$ and $2d = \frac{1}{2} m \dim_\bbQ(H)$.
In particular, for $i = -2n,\ldots,2n$ we get an embedding
of Hodge structures
$$
H^{i+2n}(X,\bbQ(n))\hrarr \Lambda^{i+2d}(H^{* \oplus m})(d).
$$
\end{thm}

\begin{rem} The shifts in the cohomological grading in the above statement are necessary
to make the grading compatible with the action of the element $\Xi\in \frso(\tilde{V},\tilde{q})$.
The statement about the embedding of Hodge structures includes the appropriate Tate twists.
\end{rem}

\subsection{Main result} We go back to the degeneration $\pi\colon \XX\to \Delta$.
The monodromy acting on $H^2(X,\bbQ)$ is $\gamma = e^N$, where $N\in \frso(V,q)\subset\frso(\tilde{V},\tilde{q})$.
Let us denote by $\delta\in \GL(H^\sdot(X,\bbQ))$ the monodromy operator for the full cohomology
algebra.

\begin{prop}\label{prop_mon}
There exists an integer $k>0$, such that $\delta^k = e^{kN}$, where $e^{kN}$
acts on $H^\sdot(X,\bbQ)$ via the representation $\Spin(V,q)\to \mathrm{Aut}^P(X)$.
\end{prop}
\begin{proof}
The monodromy operator $\delta$ is induced by a diffeomorphism of $\XX_t$ for a base point $t\in\Delta^*$.
Hence it is contained in the arithmetic subgroup $\Gamma_A = \mathrm{Aut}^P(X)_\bbZ$ of $\mathrm{Aut}^P(X)_\bbQ$.
The claim follows from Lemma \ref{lem_mc} applied to this subgroup.
\end{proof}

We get the following immediate consequence, that recovers Corollary 3.2 from \cite{KLSV}:

\begin{cor} If the monodromy action on $H^2(X,\bbQ)$ is trivial, then its action
on $H^\sdot(X,\bbQ)$ is of finite order.
\end{cor}

Next we compare the limit MHS on $X$ and the Kuga-Satake abelian variety.

\begin{prop}\label{prop_lim_MHS}
There exists an integer $m>0$ and an embedding of mixed Hodge structures
$$(H^{\sdot+2n}(X,\bbQ(n)),\tilde{W}_\sdot,\tilde{F}^\sdot_{\mathrm{lim}})\hrarr (\Lambda^{\sdot+2d}(H^{*\oplus m})(d), W_\sdot,F^\sdot_{\mathrm{lim}}),$$
where $2n = \dim_\bbC(X)$ and $2d = \frac{1}{2} m \dim_\bbQ(H)$.
\end{prop}
\begin{proof}
We use the same convention with the shift of cohomological grading as in Theorem \ref{thm_ks},
see the remark after that theorem. In particular, the Hodge filtration on $H^{\sdot+2n}(X,\bbQ(n))$
has non-trivial graded components in degrees $-n,\ldots,n$.

The limit mixed Hodge structures do not change if we replace the monodromy operator by its power.
Thus we may use Proposition \ref{prop_mon} and assume that $\delta = e^N$, where the exponential
is viewed as an element of $\Spin(V,q)$. This implies that the embedding from Theorem \ref{thm_ks}
is compatible with the weight filtrations, since they are both induced by the action of $N\in \frso(V,q)$.

Next we deal with the limit Hodge filtrations. Let us denote by $\DD_X$ and $\hat{\DD}_X$ the period domain,
respectively the extended period domain for the $h$-polarized Hodge structures on $H^\sdot(X,\bbQ)$.
Analogously, $\DD_{KS}$ and $\hat{\DD}_{KS}$ will denote the period domain, respectively the extended
period domain for the Hodge structures on $\Lambda^\sdot(H^{*\oplus m})$ polarized by a fixed form $\omega$
as above. Both $\hat{\DD}_X$ and $\hat{\DD}_{KS}$ are closed subvarieties of certain flag
varieties, and $\DD_X$, $\DD_{KS}$ are their open subsets (see \cite{Sch} for the description of period domains
as subvarieties of flag varieties).

The variety $\hat{\DD}_X$ carries a universal family of holomorphic bundles that determine the Hodge filtration:
$$
H^\sdot(X,\bbQ)\otimes\OO_{\hat{\DD}_X} = \tilde{\FF}^{-n}\supset \tilde{\FF}^{-n+1} \supset\ldots\supset \tilde{\FF}^{n}\supset \tilde{\FF}^{n+1} = 0.
$$
Analogously, over $\hat{\DD}_{KS}$ we have a family of subbundles
$$
\Lambda^\sdot(H^{*\oplus m})\otimes\OO_{\hat{\DD}_{KS}} = \FF^{-d}\supset \FF^{-d+1} \supset\ldots\supset \FF^{d}\supset \FF^{d+1} = 0.
$$

Let $p_X$ and $p_{KS}$ denote the two projections from
$\hat{\DD}_X\times \hat{\DD}_{KS}$ to the factors. For every $i=-n,\ldots,n$ consider the
morphism of vector bundles $$\eta_i\colon p_{X}^*\tilde{\FF}^i\to p_{KS}^*(\FF^{-d}/\FF^{i})$$
obtained as the composition of three morphisms:
the embedding $p_{X}^*\tilde{\FF}^i \hrarr H^\sdot(X,\bbQ)\otimes\OO_{\hat{\DD}_X\times \hat{\DD}_{KS}}$,
the embedding from Theorem \ref{thm_ks} and the projection to the quotient $\Lambda^\sdot(H^{*\oplus m})\otimes\OO_{\hat{\DD}_X\times \hat{\DD}_{KS}} \to p_{KS}^*(\FF^{-d}/\FF^{i})$. Denote by $\mathcal{Z}$ the closed subscheme
of $\hat{\DD}_X\times \hat{\DD}_{KS}$ where all $\eta_i$ vanish. The points of $\mathcal{Z}$
correspond to such pairs of filtrations that the embedding from Theorem \ref{thm_ks} is
compatible with them.

After passing to the universal cover of the punctured disc, we get two period maps
$\tilde{\varphi}_X\colon \tilde{\Delta}\to \DD_X$ and $\tilde{\varphi}_{KS}\colon \tilde{\Delta}\to \DD_{KS}$.
Since the Hodge structures on $H^\sdot(X,\bbQ)$ and $\Lambda^\sdot(H^{*\oplus m})$ are both
determined by the action of the Weil operators $W(z)\in \frso(V,q)$, $z\in \tilde{\Delta}$,
the embedding from Theorem \ref{thm_ks} is a morphism of Hodge structures. This means
that the product of $\tilde{\varphi}_X$ and $\tilde{\varphi}_{KS}$ gives a map
$\tilde{\varphi}\colon \tilde{\Delta}\to \mathcal{Z}\subset \hat{\DD}_X\times \hat{\DD}_{KS}$.

Consider now the twisted period maps $\tilde{\psi}_X\colon \tilde{\Delta}\to \hat{\DD}_X$ and
$\tilde{\psi}_{KS}\colon \tilde{\Delta}\to \hat{\DD}_{KS}$, where $\tilde{\psi}_X(z) = e^{-zN}\tilde{\varphi}_X(z)$ and
$\tilde{\psi}_{KS}(z) = e^{-zN}\tilde{\varphi}_{KS}(z)$. Let $\tilde{\psi}\colon \tilde{\Delta}\to
\hat{\DD}_X\times \hat{\DD}_{KS}$ be their product. Since the subscheme $\mathcal{Z}$
is $\Spin(V,q)$-invariant by construction, we have $\tilde{\psi}\colon \tilde{\Delta}\to \mathcal{Z}$.

By a theorem of Schmid \cite[Theorem 4.9]{Sch}
there exists a limit $\lim\limits_{\mathrm{Im}z\to +\infty}\tilde{\psi}(z)$.
The corresponding point of $\hat{\DD}_X\times \hat{\DD}_{KS}$ determines the pair
of limit Hodge filtrations $\tilde{F}^\sdot_{\mathrm{lim}}$ and $F^\sdot_{\mathrm{lim}}$.
Since the subscheme $\mathcal{Z}$ is closed, the limit
lies in $\mathcal{Z}$. We conclude that the limit Hodge filtrations are compatible.
\end{proof}

\begin{thm}\label{thm_main} Let $X$ be a simple hyperk\"ahler manifold and let
$\pi\colon \XX\to \Delta$ be a maximally unipotent projective degeneration of $X$.
Then the limit mixed Hodge structures on $H^k(X,\bbQ)$ are of
Hodge-Tate type for all $k$.
\end{thm}
\begin{proof}
Mixed Hodge structures of Hodge-Tate type form a tensor subcategory inside the abelian category of MHS.
Hence by Proposition \ref{prop_lim_MHS} it suffices to check that the limit MHS of the
variation of Kuga-Satake Hodge structures $(H,W_\sdot,F^\sdot_{\mathrm{lim}})$ is of Hodge-Tate type.
This follows from \cite[proof of Theorem 1.2(3)]{SS}. 
\end{proof}

The knowledge of the limit MHS provides some information about cohomology of
the central fibre, at least if the degeneration is semistable.

Let $Y$ be a quasi-projective variety. Recall, that cohomology groups of $Y$ carry functorial
mixed Hodge structures. If $Y$ is projective, the weights of the MHS on $H^k(Y,\bbQ)$
lie in the range $0,\ldots,k$ for $k\le\dim_\bbC(Y)$ and $2k - 2\dim_\bbC(Y),\ldots, k$ for $k > \dim_\bbC(Y)$.

\begin{defn} We will say that the mixed Hodge structure on $H^k(Y,\bbQ)$ is semi-pure, if the induced
mixed Hodge structure on $W_{k-1}H^k(Y,\bbQ)$ is of Hodge-Tate type.
\end{defn}

\begin{cor} In the setting of Theorem \ref{thm_main}, assume that
$\XX_0$ is a reduced divisor with simple normal crossings. Then the mixed Hodge structures
on $H^k(\XX_0,\bbQ)$ are semi-pure for all $k$.
\end{cor}
\begin{proof}
Consider the Clemens-Schmid exact sequence of MHS, where $\nu = (2\pi \ii)^{-1}N$:
$$
\ldots\lrarr H^k_{\XX_0}(\XX,\bbQ)\lrarr H^k(\XX_0,\bbQ)\lrarr H^k_{\mathrm{lim}}(X,\bbQ) \stackrel{\nu}{\lrarr} H^k_{\mathrm{lim}}(X,\bbQ)(-1)\lrarr\ldots
$$
It follows from Poincar\'e duality that the MHS on $H^k_{\XX_0}(\XX,\bbQ)$ has weights $\ge k$.
Hence the MHS on $W_{k-1}H^k(\XX_0,\bbQ)$ is determined by the limit
MHS of the degeneration. The claim now follows from Theorem \ref{thm_main}.
\end{proof}

\subsection{Unipotency indices of the monodromy action on higher cohomology groups}

It was observed in \cite[Proposition 6.18]{KLSV}, that for maximally unipotent degenerations
of hyperk\"ahler manifolds the index of unipotency of the monodromy
action on $H^{2k}(X,\bbQ)$ equals $2k+1$, where $k = 1,\ldots, n$ and as before $2n = \dim_\bbC(X)$.
We will explain below, that it is also possible to determine the index of unipotency
for odd degree cohomology groups, see Proposition \ref{prop_index}.
This applies, in particular, to maximal degenerations of generalized Kummer
type manifolds.

In this subsection we will briefly write $H^\sdot$ for $H^\sdot(X,\bbC)$ considered
as an $\frso(\tilde{V}_\bbC,\tilde{q})$-module (see section \ref{subsec_mukai}).
We will use the highest weight theory for the orthogonal Lie algebra (see e.g. \cite[Chapter VIII, \S 13]{Bou}).
Let us fix two elements $\xi_0, \xi_1\in \frso(\tilde{V}_\bbC,\tilde{q})$ that define the Hodge
bigrading on $H^\sdot$. More precisely, $\xi_0$ acts on $H^{r,s}$ as multiplication by $\frac{1}{2}(r+s) - n$,
and $\xi_1$ as multiplication by $\frac{1}{2}(s-r)$.
Next we choose a Cartan subalgebra $\tilde{\frh}\subset \frso(\tilde{V}_\bbC,\tilde{q})$
that contains these two elements and fix a basis $\tilde{\frh} = \langle \xi_0,\xi_1,\ldots,\xi_l\rangle$,
where $l = \lfloor\frac{1}{2}\dim V\rfloor$.
Note that $\frh = \tilde{\frh}\cap \frso(V_\bbC,q) = \langle \xi_1,\ldots,\xi_l\rangle$ is
a Cartan subalgebra of $\frso(V_\bbC,q)$. Let $\varepsilon_i$ denote
the dual basis: $\tilde{\frh}^* = \langle \varepsilon_0,\ldots,\varepsilon_l\rangle$.

We recall from loc. cit. the expressions for positive roots and fundamental weights.
In the case of odd $\dim V$, the set of positive roots in $\tilde{\frh}^*$ is
$R^+=\{\varepsilon_i\st 0\le i\le l\}\cup\{\varepsilon_i \pm \varepsilon_j\st 0\le i< j\le l\}$;
the fundamental weights are: $\varpi_i = \varepsilon_0+\ldots+\varepsilon_i$, $i = 0,\ldots,l-1$
and $\varpi_l = \frac{1}{2}(\varepsilon_0+\ldots+\varepsilon_l)$. The representation
with highest weight $\varpi_l$ is the spinor representation. 

In the case when $\dim V$ is even, we always have $l\ge 2$, since $b_2(X)\ge 3$.
Then $R^+=\{\varepsilon_i \pm \varepsilon_j\st 0\le i< j\le l\}$;
the fundamental weights are: $\varpi_i = \varepsilon_0+\ldots+\varepsilon_i$, $i = 0,\ldots,l-2$,
$\varpi_{l-1} = \frac{1}{2}(\varepsilon_0+\ldots+\varepsilon_{l-1}-\varepsilon_l)$
and $\varpi_l = \frac{1}{2}(\varepsilon_0+\ldots+\varepsilon_{l-1}+\varepsilon_l)$. The representations
with highest weights $\varpi_{l-1}$ and $\varpi_l$ are the two semi-spinor representations.

The images of $\varpi_1,\ldots,\varpi_l$ under the natural projection $\tilde{\frh}^*\to \frh^*$
are the fundamental weights of $\frso(V_\bbC,q)$. We will denote them by the same letters.

One can determine the highest weight of the irreducible $\frso(\tilde{V}_\bbC,\tilde{q})$-submodule of $H^\sdot$
generated by $H^0$. This submodule coincides with the subalgebra of $H^\sdot$ generated by $H^2$, whose description
is well-known, see e.g \cite{Ve1}. The action of $\frso(V_\bbC,q)$ on $H^0$ is trivial, and
the element $\xi_0$ acts as the scalar $-n$. It is clear that $H^0$ is spanned by a lowest
weight vector with weight $-n\varepsilon_0$. So the subrepresentation generated by $H^0$
is of highest weight $n\varepsilon_0$; it can be described as the kernel of the map
$S^n\tilde{V}_\bbC\to S^{n-2}\tilde{V}_\bbC$ given by contraction with $\tilde{q}$.
Next we would like to determine the possible highest weights of the subrepresentation generated
by $H^3$.

\begin{lem}\label{lem_H3}
The $\frso(V_\bbC,q)$-module $H^3$ is the direct sum of several
copies of spinor or semi-spinor representations.
\end{lem}
\begin{proof} We assume that $H^3\neq 0$.
Since $H^{3,0} = H^{0,3} = 0$, the vector space $H^3$ is the direct sum of two eigenspaces of $\xi_1$
with eigenvalues $\pm\frac{1}{2}$. Assume that $H^3$ contains an irreducible subrepresentation
with highest weight $a_1\varpi_1 + \ldots + a_l\varpi_l$, $a_i\in \bbZ_{\ge 0}$.
If $\dim V$ is odd, $\xi_1$ acts on the highest weight vector as the scalar $a_1+\ldots+a_{l-1}+\frac{1}{2}a_l$.
This is only possible when $a_1=\ldots=a_{l-1}=0$ and $a_l=1$.
If $\dim V$ is even, $\xi_1$ acts as $a_1+\ldots+a_{l-2}+\frac{1}{2}(a_{l-1}+a_l)$.
This is only possible when $a_1=\ldots=a_{l-2}=0$ and either $a_{l-1}=1$, $a_l=0$, or $a_{l-1}=0$, $a_l=1$.
In both cases we have either spinor or semi-spinor representation.
\end{proof}

\begin{lem}\label{lem_irred}
Let $W^\sdot\subset H^\sdot$ be an irreducible $\frso(\tilde{V}_\bbC,\tilde{q})$-submodule.
Then $W^3 \simeq W^{4n-3}$ is an irreducible $\frso(V_\bbC,q)$-module.
\end{lem}
\begin{proof}
We can assume that $W^3\neq 0$. Then $W^{2k} = 0$ for all $k$, since otherwise $W^{2\sdot}$
would be a non-trivial subrepresentation of $W^\sdot$. We also have $W^1 = H^1 = 0$, and it
follows that the minimal eigenvalue of $\xi_0$ is $\frac{3}{2}-n$, with $W^3$ being the corresponding eigenspace.
We can find a vector $w\in W^3$ that is of lowest weight with respect to $\frso(V_\bbC,q)$ and $\frh$.
Then $w$ is also of lowest weight for $\frso(\tilde{V}_\bbC,\tilde{q})$ and $\tilde{\frh}$. By our assumption,
$w$ is unique up to multiplication by a scalar. This implies irreducibility of $W^3$. The
isomorphism $W^3\simeq W^{4n-3}$ follows from Poincar\'e duality.
\end{proof}

\begin{lem}\label{lem_weight}
Let $W^\sdot\subset H^\sdot$ be an irreducible $\frso(\tilde{V}_\bbC,\tilde{q})$-submodule,
such that $W^3\neq 0$. Then the highest weight $\mu$ of $W^\sdot$ is one of the following.
If $\dim(V)$ is odd, then $\mu = (n-2)\varpi_0 + \varpi_l$. If $\dim(V)$ is even, then
either $\mu = (n-2)\varpi_0 + \varpi_{l-1}$ or $\mu = (n-2)\varpi_0 + \varpi_l$.
\end{lem}
\begin{proof}
It follows from Lemma \ref{lem_H3} and Lemma \ref{lem_irred} that $W^{4n-3}$ is either
spinor or semi-spinor representation of $\frso(V_\bbC,q)$. Thus $\mu$ is either of the form
$k\varpi_0 + \varpi_l$ or $k\varpi_0 + \varpi_{l-1}$ (when $\dim(V)$ is even), for some $k$.
Then $\xi_0$ acts on the highest weight vector as $k+\frac{1}{2}$, and since
the highest weight vector is contained in $W^{4n-3}$, we have $k=n-2$.
\end{proof}

\begin{lem}\label{lem_weight2}
Assume that $H^3(X,\bbC)\neq 0$. Then $H^{2k+1}(X,\bbC)$
for $k = 1,\ldots, n-1$ contains an $\frso(V_\bbC,q)$-submodule of highest weight $\nu$,
which can be one of the following. If $\dim(V)$ is odd, then $\nu = (k-1)\varpi_1 + \varpi_l$;
if $\dim(V)$ is even, then either $\nu = (k-1)\varpi_1 + \varpi_l$ or $\nu = (k-1)\varpi_1 + \varpi_{l-1}$.
\end{lem}
\begin{proof}
Let us assume that $\dim(V)$ is odd, the other case being analogous.
We pick an irreducible $\frso(\tilde{V}_\bbC,\tilde{q})$-submodule $W^\sdot\subset H^\sdot$ with $W^3\neq 0$.
We know from Lemma \ref{lem_weight} that the highest weight of $W^\sdot$ is
$\mu = (n-2)\varpi_0 + \varpi_l = \frac{2n-3}{2}\varepsilon_0 + \frac{1}{2}(\varepsilon_1+\ldots+\varepsilon_l)$.
The set of weights of a representation is invariant with respect to the Weyl group action.
Since the transposition of $\varepsilon_0$ and $\varepsilon_1$ belongs to the Weyl group,
the weight $\mu' = \frac{1}{2}\varepsilon_0 + \frac{2n-3}{2}\varepsilon_1 + \frac{1}{2}(\varepsilon_2+\ldots+\varepsilon_l)$ also belongs to $W^\sdot$.
Let $\alpha = \varepsilon_0 - \varepsilon_1$ be one of the positive roots.
Then $\mu' = \mu - (n-2)\alpha$. Consider the action of the $\frsl_2$-subalgebra corresponding to the 
root $\alpha$. It follows from the representation theory of $\frsl_2$, that all the weights
of the form $\mu - i\alpha$, $i = 0,\ldots, n-2$ belong to $W^\sdot$.
The corresponding weight subspaces are contained in $H^{4n-3-2i}$. By restricting
to $\frso(V_\bbC,q)$ we find that $H^{3+2i}\simeq H^{4n-3-2i}$ contains a subrepresentation with highest
weight $\frac{2i+1}{2}\varepsilon_1 + \frac{1}{2}(\varepsilon_2 + \ldots + \varepsilon_l)$.
Setting $k = i+1$ we get the result.
\end{proof}

\begin{prop}\label{prop_index}
Assume that $H^3(X,\bbQ)\neq 0$. Consider a maximal degeneration of $X$,
and let $\NN$ denote the logarithm of the monodromy acting on $H^{2k+1}(X,\bbQ)$,
where $k = 1,\ldots, n-1$.
Then $\NN^{2k-1}\neq 0$, $\NN^{2k} = 0$.
\end{prop}
\begin{proof}
The fact that $\NN^{2k} = 0$ follows from the general result of Schmid \cite[Theorem 6.1]{Sch}
and the vanishing of Hodge numbers $h^{2k+1,0}(X) = h^{0,2k+1}(X) = 0$.

Let $N$ denote the logarithm of the monodromy acting on $H^2(X,\bbQ)$.
According to Proposition \ref{prop_mon}, we may assume that $\NN$ is the image
of $N$ under the homomorphism $\frso(V,\bbQ)\to \emrp(H^{2k+1}(X,\bbQ))$ (see section \ref{subsec_mukai}).
Let us assume that $\dim(V)$ is odd.
By Lemma \ref{lem_weight2}, it is enough to consider the representation
of highest weight $(k-1)\varpi_1 + \varpi_l$, 
and to prove that $N^{2k-1}$ acts non-trivially on it.

We can choose two isotropic subspaces $U = \langle e_1, \ldots, e_l\rangle\subset V_\bbC$
and $U' = \langle e_1', \ldots, e_l'\rangle\subset V_\bbC$ and an element $e_{l+1}$ orthogonal
to them, so that $q(e_i,e_j') = 0$ for $1\le i < j\le l$, $q(e_i,e_i') = 1$, $q(e_{l+1},e_{l+1}) = 1$
and $V_\bbC = U\oplus U'\oplus \langle e_{l+1}\rangle$.
We may moreover assume (see \cite[proof of Proposition 4.1]{SS}) that this decomposition is
compatible with $N$ in the sense that $N = e_1'\wedge (e_2+e_2')$, where we use the identification
$\frso(V_\bbC,q) \simeq \Lambda^2V_\bbC$.
We also choose the Cartan subalgebra of $\frso(V_\bbC,q)$ corresponding to this decomposition
(see \cite[Chapter VIII, \S 13]{Bou}).

Denote by $\PP^i$ the $\frso(V_\bbC,q)$-module of highest weight $i\varpi_1$.
Then $\PP^i$ is a subrepresentation of $S^iV_\bbC$, and it is generated
by the highest weight vector $e_1^i$. Note that $Ne_1 = -e_2-e_2'$, 
$N^{2}e_1 = -2e_1'$ and by Leibnitz's rule $N^{2i}(e_1^i)\neq 0$.

Let $\mathcal{S}$ be the spinor representation. It can be described as $\Lambda^\sdot U$,
on which $e_i$ act by exterior multiplication and $e_i'$ act by contraction, $i=1,\ldots, l$.
The highest weight vector is $u = e_1\wedge\ldots\wedge e_l\in \Lambda^\sdot U$,
and we see that $Nu = e_3\wedge\ldots\wedge e_l\neq 0$.

The element $e_1^{k-1}\otimes u \in \PP^{k-1}\otimes \mathcal{S}$ has
weight $(k-1)\varpi_1 + \varpi_l$, hence it generates the representation
we are interested in.
Leibnitz's rule again implies that $N^{2k-1}(e_1^{k-1}\otimes u) \neq 0$.
This proves the claim for $\dim(V)$ odd. The case of even dimension is analogous.
\end{proof}

\section{Existence of degenerations with maximal unipotent monodromy}\label{sec_existence}

In this section we fix a hyperk\"ahler manifold $X$ as in section \ref{sec_notations}, and
assume moreover that $b_2(X) \ge 5$. This condition is satisfied for all known
families of hyperk\"ahler manifolds. Our goal is to show that $X$ admits a projective
degeneration in the sense of Definition \ref{def_deg}, such that the monodromy
operator $\gamma\in O(V_\bbZ, q)$ is of the form $\gamma = e^N$, $N\in \frso(V,q)$ with $N^2\neq 0$,
$N^3 = 0$.

The construction consists of two steps. First, we find a nilpotent operator $N$ that
satisfies the above conditions and prove that there exist sufficiently many nilpotent orbits
(see Definition \ref{def_nilp}). Second, we show that one can find a nilpotent orbit
that corresponds to a projective degeneration of $X$.

\subsection{Nilpotent orbit} We fix $V_\bbZ$, $V$, $q$ as in section \ref{sec_notations}.
Recall that the signature of $q$ is $(3,\dim(V)-3)$. Given an element $h\in V$, $V^h$ denotes
its orthogonal complement.

\begin{lem}\label{lem_N}
Assume that $\dim(V)\ge 5$. Then there exist an element $h\in V_\bbZ$ with $q(h)>0$
and an endomorphism $N\in \frso(V^h,q)$, such that $N^2\neq 0$, $N^3 = 0$ and the
restriction of $q$ to $\mathrm{Im}(N)$ is semi-positive with one-dimensional kernel.
\end{lem}
\begin{proof}
By Meyer's theorem there exists a vector $v_0\in V$, such that $q(v_0)=0$. An elementary
argument shows that $v_0$ is contained in a hyperbolic plane: $v_0 = 1/2(v_1+v_2)$
for some $v_1, v_2\in V$ with $q(v_1) = 1$, $q(v_2) = -1$, $q(v_1,v_2) = 0$.
The restriction of $q$ to the subspace $V' = \langle v_1,v_2 \rangle^\perp$ has signature $(2,\dim(V)-4)$.
Thus we can find elements $v_3, h\in V'$ with $q(v_3)>0$, $q(h)>0$ and $q(v_3,h) = 0$.
We may moreover assume that $h\in V_\bbZ$.

Recall the natural isomorphism $\frso(V^h,q)\simeq \Lambda^2V^h$. Let $N$ correspond to
the bivector $v_0\wedge v_3$ under this isomorphism. We see easily that $\mathrm{Im}(N) = \langle v_0,v_3\rangle$,
$\mathrm{Im}(N^2) = \langle v_0\rangle$ and $N^3 = 0$.
\end{proof}

\begin{rem}
One can show that for any degeneration of K3-type Hodge structures with unipotent mo\-no\-dromy the
nilpotent operator $N$ is of the same form as in the proof above (see \cite[Proposition 4.1]{SS}).
For the other type of degenerations (such that $N\neq 0$, $N^2=0$) the operator $N$ is
given by $w_1\wedge w_2$, where $\langle w_1,w_2 \rangle\subset V$ is an isotropic subspace.
It is clear that such subspaces exist whenever $\dim(V)\ge 6$.
\end{rem}

We recall the definition of a nilpotent orbit from \cite{Sch}. The condition of Griffiths
transversality is satisfied automatically in our case, so we do not include it.

\begin{defn}\label{def_nilp}
Let $N\in \frso(V^h,q)$ be nilpotent and $x\in \hat{\DD}^h$. The pair $(N,x)$ defines a nilpotent orbit
if there exists $t_0\ge 0$, such that $e^{\ii t N}x\in \DD^h$ for all $t> t_0$. The corresponding
nilpotent orbit is the image of $\bbC$ in $\hat{\DD}^h$ under the map $z\mapsto e^{zN}x$.
\end{defn}

\begin{lem}\label{lem_nilp_orb}
Fix $N$ and $h$ as in Lemma \ref{lem_N}. For $x\in \hat{\DD}^h$ the pair $(N,x)$ defines
a nilpotent orbit if and only if $q(Nx,N\bar{x})>0$. The set of such points in $\hat{\DD}^h$ is
open and non-empty.
\end{lem}
\begin{proof}
The condition that should be satisfied is the following: $q(e^{\ii t N}x, e^{-\ii t N}\bar{x})>0$,
or equivalently $q(e^{2\ii t N}x, \bar{x})>0$ for $t\gg 0$. We have $e^{2\ii t N} = 1 + 2 \ii t N - 2 t^2 N^2$,
and our condition is equivalent to $-q(N^2x,\bar{x}) = q(Nx,N\bar{x})>0$.
The set of such $x\in \hat{\DD}^h$ is clearly open. It is non-empty, because
$q$ is semi-positive on $\mathrm{Im}(N)$. 
\end{proof}

\subsection{Projective degeneration} Our aim now is to prove that some
of the nilpotent orbits from Lemma \ref{lem_nilp_orb} are induced by
projective degenerations of $X$, i.e. obtained as the period map
for such degenerations.

Let us fix $N$ and $h$ as in Lemma \ref{lem_N}, and let $\Gamma\subset \mathrm{O}(V_\bbZ,q)$
be a torsion-free arithmetic subgroup. The following lemma is well-known
to the experts and admits many different proofs. For completeness we sketch
a proof via Hilbert schemes.

\begin{lem}\label{lem_family}
There exists a projective family $\varphi\colon \YY\to S$ of hyperk\"ahler
manifolds deformation equivalent to $X$ over a smooth quasi-projective base $S$,
and a commutative diagram
\begin{equation}\label{diagram}
\begin{tikzcd}[]
\tilde{S} \dar{q}\rar{\tilde{\rho}} & \DD^h\dar{p} \\
S \rar{\rho} & \DD^h/\Gamma
\end{tikzcd}
\end{equation}
In this diagram: $\tilde{S}$ is the universal covering of $S$, $\tilde{\rho}$
is the period map for the family $q^*\YY$, and $\rho$ is an \'etale
algebraic morphism, such that $S$ is a finite unramified covering of $\rho(S)$.
\end{lem}

\begin{proof} Surjectivity of the period map for hyperk\"ahler manifolds implies
that we can find a deformation of $X$ whose period is contained in $\DD^h$. Let $Y$
be such a deformation. We can also assume that the Picard group of $Y$ has rank one.
We fix an isomorphism $H^2(Y,\bbZ)\simeq V_\bbZ$, so that $h$ generates $\mathrm{Pic}(Y)$.

The manifold $Y$ is projective by \cite[Theorem 3.11]{Hu1} (see erratum to that paper for the correct proof),
and after replacing $h$ by its
multiple we can assume that $h = [L]$ for a very ample line bundle $L\in\mathrm{Pic}(Y)$,
such that $H^i(Y,L) = 0$ for all $i>0$. Let $W = H^0(Y,L)$ and
consider the embedding of $Y$ into $\bbP(W^*)$. This determines a $\bbC$-point $[Y]$ in the
Hilbert scheme $\mathrm{Hilb}(\bbP(W^*)/\bbC)$. Let $\HH$ be an irreducible component containing
this point, and $\phi\colon \YY\to \HH$ the universal family. Let $S'\subset \HH^{\mathrm{red}}$
be the maximal Zariski-open subset over which the restriction of $\phi$ to $\HH^{\mathrm{red}}$ is smooth;
it is non-empty since it contains the point $[Y]$. The family $\YY$ induces a variation of
Hodge structures on the local system $\VV = (R^2\phi_*\bbZ)_{\mathrm{pr}}$ with fibre $V_\bbZ^h$ over $S'$.

Consider the universal covering $q'\colon \tilde{S'}\to S'$. The pull-back $q'^*\VV$ induces
the period map $\rho'\colon \tilde{S'}\to \DD^h$ and the monodromy representation
$\mu\colon \pi_1(S')\to \mathrm{O}(V_\bbZ^h,q)$. Consider the group $\Gamma' = \mu^{-1}(\Gamma)$.
It is of finite index in $\pi_1(S')$ and we have the corresponding finite covering $S''$.
By construction, the morphism $\tilde{\rho}'$ descends to $\rho''\colon S''\to \DD^h/\Gamma$.

We claim that $\rho''$ is dominant. To see this, one can consider the universal deformation
of the pair $(Y,L)$. By the local Torelli theorem the base of this deformation is an open
subset $B\subset \DD^h$ isomorphic to a polydisc. The universal family over $B$ induces, possibly after
shrinking $B$, a morphism $B\to S'$ that lifts to $B\to S''$ since $B$ is simply-connected.
The composition with $\rho''$ gives a map $B\to \DD^h/\Gamma$, that by construction
equals the composition $B\hrarr \DD^h\to \DD^h/\Gamma$. The image of $\rho''$ thus
contains an open subset of $\DD^h/\Gamma$, so $\rho''$ is dominant.

Finally, to construct $S$ that satisfies all the conditions of the lemma, one can take
a subvariety of $S''$ that has the same dimension as $\DD^h/\Gamma$ and maps dominantly to it. Restriction of $\rho''$
to some open subset of such subvariety will be \'etale. By shrinking this open subset further
we can make $\rho''$ finite over its image. We let $S$ be the resulting locally closed subvariety of $S''$.
The preimage $q'^{-1}(S)$ is a covering of $S$. We can replace it by the universal covering $\tilde{S}$,
and the period map $\rho$ then factors through $q'^{-1}(S)$.
\end{proof}

\begin{thm}\label{thm_existence}
Let $X$ be a simple hyperk\"ahler manifold with $b_2(X)\ge 5$. Then there exists
a projective degeneration of $X$ with maximal unipotent monodromy. The limit
mixed Hodge structures on all cohomology groups of this degeneration are of Hodge-Tate type.
\end{thm}
\begin{proof}
Consider the diagram \ref{diagram}.
Let $M$ be a smooth projective variety that contains $\DD^h/\Gamma$ as an open subset.
Let $U$ be the image of $\rho$ and $D$ the complement of $U$ inside $M$ with reduced scheme structure.

Let $N$ be a nilpotent endomorphism from Lemma \ref{lem_N}. After multiplying $N$
by a positive integer, we may assume that $\gamma = e^N\in \Gamma$. 
Let $\NN = \{ x\in \DD^h\st (N,x) \mbox{ defines a nilpotent orbit}\}$. It follows from Lemma \ref{lem_nilp_orb}
and the definition of the nilpotent orbit that $\NN$ is open and non-empty. Hence
there exists a point $x_0\in \NN\cap p^{-1}(U)$.

Let us consider the nilpotent orbit given by $(N,x_0)$. Choose $t_0 > 0$ as in Definition \ref{def_nilp},
and define $\tilde{\alpha}\colon \{z\in \bbC\st \mathrm{Im}(z)>t_0\}\to \DD^h$, $z\mapsto e^{zN}x_0$.
We have $\tilde{\alpha}(z+1) = \gamma \tilde{\alpha}(z)$, so $\tilde{\alpha}$ descends to a map
$\alpha\colon \Delta_\varepsilon^*\to \DD^h/\Gamma$, where
$\Delta_\varepsilon^*$ is the punctured disc of radius $\varepsilon = e^{-t_0}$. According to Borel,
we can extend $\alpha$ over the puncture and get a map $\Delta_\varepsilon\to M$. 

Note that the image of $\alpha$ is not contained in $D$ by the choice of $x_0$. Since $\alpha^{-1}(D)$
is an analytic subvariety of $\Delta_\varepsilon$, we can find $\varepsilon'\le\varepsilon$, such
that $\alpha^{-1}(D)\cap \Delta^*_{\varepsilon'} = \emptyset$. We get a map $\alpha\colon \Delta^*_{\varepsilon'}\to U$.
After passing to a finite unramified covering of $\Delta^*_{\varepsilon'}$ and rescaling the
coordinate on the disc, we obtain a map
$\alpha'\colon \Delta^*\to S$. We can find projective compactifications
$\bar{\YY}$ and $\bar{S}$ of $\YY$ and $S$ from Lemma \ref{lem_family}, such that
$\phi$ extends to $\bar{\phi}\colon\bar{\YY}\to\bar{S}$. According to Borel and Kobayashi,
$\alpha'$ extends to a morphism $\Delta\to \bar{S}$. Then the pull-back
of $\bar{\YY}$ by $\alpha'$ defines a projective degeneration over $\Delta$.
The assertion about the limit mixed Hodge structures follows from Theorem \ref{thm_main}.
\end{proof}

\section*{Acknowledgements}

I am grateful to Misha Verbitsky and Daniel Hyubrechts for useful discussions and the reference \cite{To},
and to the referee for valuable comments and suggestions.

\end{document}